\documentclass[12pt]{article}

\usepackage{amsmath,amsthm,amsfonts,latexsym,amscd,amssymb,}
\def\qed{{\hbadness=10000\hfill\ \vbox{\hrule height.09ex
     \hbox{\vrule width.09ex height1.55ex depth.2ex \kern1.8ex
     \vrule width.09ex height1.55ex depth.2ex}\hrule height.09ex}\break
     \bigskip}}

\newtheorem{thm}{Theorem}[section]
\newtheorem{prop}[thm]{Proposition}
\newtheorem{cor}[thm]{Corollary}
\newtheorem{lem}[thm]{Lemma}

\newtheorem{defn}[thm]{Definition}

\newtheorem{rem}[thm]{Remark}
\numberwithin{equation}{section}

\newtheorem{example}[thm]{Example}

\begin{document}

\title{Solvable and Nilpotent Right Loops}

\author{Vivek Kumar Jain*
~ and Vipul Kakkar**\\
*Department of Mathematics, Central University of Bihar \\
Patna (India) 800014\\
**School of Mathematics, Harish-Chandra Research Institute\\
Allahabad (India) 211019\\
Email: jaijinenedra@gmail.com; vplkakkar@gmail.com}

\date{}
\maketitle


\begin{abstract}
In this paper the notion of nilpotent right transversal and solvable right transversal has been defined. Further, it is proved that if a core-free subgroup has a generating solvable transversal or a generating nilpotent transversal, then the whole group is solvable.  
\end{abstract}



\section{Introduction}
Transversals play an important role in characterizing the group and the embedding of subgroup in the group. In \cite{tm}, Tarski monsters has been characterized with the help of transversals.  In \cite{p}, Lal shows that if all subgroups of a finitely generated solvable group are stable (any right transversals of a subgroup have isomorphic group torsion), then the group is nilpotent. In \cite{ict2}, it has been shown that if the isomorphism classes of transversals of a subgroup in a finite group is $3$, then the subgroup itself has index $3$. Converse of this fact is also true.


Let $G$ be a group and $H$ a proper subgroup of $G$. A {\textit{right transversal}} is a subset of $G$ obtained by selecting one and only one element from each right coset of $H$ in $G$ and identity from the coset $H$. Now we will call it {\textit {transversal}} in place of right transversal. Suppose that $S$ is a transversal of $H$ in $G$. We define an operation $\circ $ on $S$ as follows: for $x,y \in S$, $\{x \circ y \}:=S \cap Hxy$.
It is easy to check that $(S, \circ )$ is a  right loop, that is the equation of the type $X \circ a=b$, $X$ is unknown and $a,b \in S$ has unique solution in $S$ and $(S, \circ)$ has both sided identity.
 In \cite{rltr}, it has been shown that for each right loop there exists a pair $(G,H)$ such that $H$ is a core-free subgroup of the group $G$ and the given right loop can be identified with a tansversal of $H$ in $G$. Not all transversals of a subgroup generate the group. But it is proved by Cameron in \cite{pjc}, that if a subgroup is core-free, then always there exists a transversal which generates the whole group. We call such a transversal as {\it generating transversal}.

Let $(S, \circ)$ be a right loop (identity denoted by $1$). Let $x,y,z \in S$. Define a map $f^S(y,z)$ from $S$ to $S$ as follows: $f^S(y,z)(x)$ is the unique solution of the equation $X \circ (y \circ z )=(x \circ y ) \circ z$, where $X$ is unknown. It is easy to verify that $f^S(y,z)$ is a bijective map. For a set $X$, let $Sym(X)$ denote the symmetric group on $X$. We denote by $G_S$ the subgroup of $Sym(S)$, generated by the set $\{f^S(y,z) \mid y,z \in S \}$. This group is called {\textit{group torsion}} of $S$ \cite[Definition 3.1, p. 75]{rltr}. It measures the deviation of a right loop from being a group. Our convention for the product in
the symmetric group $Sym(S)$ is given as $(rs)(x) =s(r(x))$ for $r, s \in Sym(S)$ and $x \in S$. Further, the group $G_S$  acts on $S$ through the action $\theta^S$ defined as: for $x \in S$ and $h \in G_S$; $x \theta^S h := h(x)$. Also note that right multiplication by an element of $S$ gives a bijective map from $S$ to $S$, that is an element of $Sym(S)$. 
The subgroup generated by this type of elements in $Sym(S)$ is denoted by $G_SS$ because $G_S$ is a subgroup of it and the right multiplication map associated with the elements of $S$ form a transversal of $G_S$ in $G_SS$. Note that if $H$ is a core-free subgroup of a group $G$ and $S$ is a generating transversal of $H$ in $G$, then $G\cong G_SS$ such that $H \cong G_S$ \cite[Proposition 3.10, p. 80]{rltr}.

A non-empty subset $T$ of right loop $S$ is called a \textit{right subloop} of $S$, if it is right loop with respect to induced binary operation on $T$ (see \cite[Definition 2.1, p. 2683]{rps}). An equivalence relation $R$ on a right loop $S$ is called a congruence in $S$, if it is a sub right loop of $S \times S$. 
Also an \textit{invariant right subloop} of a right loop $S$ is precisely the equivalence class of the identity of a congruence in $S$ (\cite[Definition 2.8, p. 2689]{rps}).
It is observed in the proof of \cite[Proposition 2.10, p. 2690]{rps} that if $T$ is an invariant right subloop of $S$, then the set $S/T=\{T \circ x|x \in S\}$ becomes right loop called as \textit{quotient of S mod T}. Let $R$ be the congruence associated to an invariant right subloop $T$ of $S$. Then we also denote $S/T$ by $S/R$. 

\section{Some Properties of Right Loop}
In this section, we will recall some basic facts about right loops and also prove some of the results which will be used in next sections. An invariant right subloop is playing the same role in the theory of right loops as the normal subgroup is playing in the theory of groups. 
    
The following Lemma is a part of \cite[Theorem 2.7, p. 2686]{rps}.

\begin{lem}\label{rps1}
Let $S$ be a right loop. A  right subloop $T\neq \{ 1\}$ is invariant if and only if $R=\{(x \circ y, y ) \mid x \in T, ~ \text ~ y \in S\}$ is a congruence in $S$.
\end{lem}

The discussion following \cite[Lemma 2.5, p. 2684]{rps} shows the following result.
\begin{lem} \label{lem}
An onto homomorphism $\rho : S \rightarrow S^{\prime}$ of right loops defines an onto group homomorphism from $G_S$ to $G_{S^{\prime}}$ given by $f^S(y,z) \mapsto f^{S^{\prime}}(\rho(y),\rho(z))$. 
\end{lem}

Suppose that $T$ is an invariant right subloop of $S$. It is easy to observe that the map $\nu:S\rightarrow S/T$ defined by $\nu(x)=T \circ x$ is a onto right loop homomorphism, that is $\nu (x \circ y)=\nu (x) \circ \nu (y)$ for all $x,y \in S$. By Lemma \ref{lem}, $\nu:S\rightarrow S/T$ induces an onto group homomorphism $\theta:G_S\rightarrow G_{S/T}$ defined by $f^S(y,z) \mapsto f^{S/T}(T \circ y, T \circ z)$. One can note that the kernel $Ker \theta =\{h \in G_S| T \circ h(x)= T \circ x, ~\text{for~all~} x \in S\}$. One also notes that $G_S/Ker \theta \cong G_{S/T}$.

\begin{lem}\label{prop1}
 Let $\rho$ be an onto homomorphism from a right loop $S$ to a right loop $S^{\prime}$. Let $T$ be an invariant right subloop of $S$. Then $\rho(T)$ is an invariant right subloop of $S^{\prime}$.
 \end{lem}
\begin{proof}
Since $\rho : S \rightarrow S^{\prime}$ is a homomorphism of right loops, for $x,y,z \in S$ 
\begin{equation}\label{eqq1} 
\rho((x \circ y) \circ z)=\rho(x \theta^S f^S(y,z) \circ (y \circ z))=\rho(x \theta^S f^S(y,z)) \circ (\rho(y) \circ \rho(z)).
\end{equation}
Again,
\begin{equation}\label{eq2}
\rho((x \circ y) \circ z)=(\rho(x) \circ \rho(y)) \circ \rho(z)=\rho(x) \theta^{S^{\prime}} f^{S^{\prime}}(\rho(y),\rho(z)) \circ (\rho(y) \circ \rho(z)).
\end{equation}
 Since $S^{\prime}$ is a right loop, Equation \ref{eqq1} and Equation \ref{eq2} implies
 \begin{equation}\label{eq3}
 \rho(x \theta^S f^S(y,z))=\rho(x) \theta^{S^{\prime}} f^{S^{\prime}}(\rho(y),\rho(z)).
 \end{equation}
 
 Suppose that $T$ is an invariant subright loop of $S$. Then by Lemma \ref{rps1}, $R=\{(x \circ y, y ) \mid x \in T, ~ \text ~ y \in S\}$ is a congruence in $S$. Since $\rho$ is an onto homomorphism satisfying Equation \ref{eq3}, so the set $\{ (\rho (x) \circ y, y ) \mid x \in T, ~ \text ~ y \in S^{\prime}\}$ is a congruence in $S^{\prime}$. Thus by Lemma \ref{rps1}, $\rho(T)$ is an invariant subright loop of $S^{\prime}$. 
   \end{proof}

\begin{lem}\label{lemm}
Let $f$ be a right loop homomorphism from a right loop $S_1$ to a right loop $S_2$. Let $T$ be an invariant right subloop of $S_2$. Then $f^{-1}(T)$ is an invariant right subloop of $S_1$.
\end{lem}
\begin{proof}
Consider a map $f \times f: S_1 \times S_1 \rightarrow S_2 \times S_2$ defined by $(x,y) \mapsto (f(x),f(y))$ where $(x,y) \in S_1 \times S_1$. Let $R$ be the congruence associated with $T$. 

For a non-empty set $X$, let $\Delta_X=\{(x,x)|x \in X\}$. Then one can check that $(f \times f)^{-1}(R)$ is a right subloop of $S_1 \times S_1$ and $\Delta_{S_1} \subseteq (f \times f)^{-1}(R)$. Hence by \cite[Proposition 4.3.2, p. 101]{smro}, $f^{-1}(T)$ is an invariant right subloop of $S_1$.
\end{proof}

Now, the following correspondance theorem is true for right loops.
\begin{prop}\label{prop2}
Let $\rho$ be an onto homomorphism from a right loop $S$ to a right loop $S^{\prime}$. Let $\psi(S)$ denote the set of all invariant right subloops of $S$ which contain the kernel of $\rho$ and $\psi(S^{\prime})$ denote the set of all invariant right subloop of $S^{\prime}$. Then $\rho$ induces a bijective map $\phi$ from $\psi(S)$ to $\psi(S^{\prime})$ defined by $\phi(T)=\rho(T)$. 
\end{prop}  
  The proof of the following Fundamental Theorem of homomorphism for right loops is as usual.
 \begin{prop} \label{prp1}
Let $\rho: S \rightarrow S^{\prime}$ be a homomorphism of right loops. Then there exists a unique injective homomorphism $\bar{\rho}:S/Ker \rho \rightarrow S^{\prime}$ such that $\bar{\rho} \circ \nu=\rho$, where $\nu : S \rightarrow S/Ker \rho$ is the natural homomorphism.  
 \end{prop}

\begin{lem}\label{gl2}
Let $(S,\circ)$ be a right loop. Let $R$ be the congruence on $S$ generated by $\{(x,x\theta^Sf^S(y,z)) \mid x,y, z \in S\}$. Then $R$ is the smallest 
congruence on $S$ such that the quotient right loop $S/R$ is a group.
\end{lem} 
 
\begin{proof} First, we will show that $S/R$ is a group. Let $1$ denote the identity of $S$. Since $R$ is a congruence on $S$, $R_1$ (equivalence class of $1$ under $R$) is an invariant subright loop of $S$. So for showing that $S/R$ is a group, it is sufficient to show that the binary operation of $S/ R$ is associative. Let $x, y, z \in S$. Then
\newline
\begin{align*}
 ((R_1\circ x)\circ (R_1\circ y))\circ (R_1\circ z) & =  (R_1\circ (x\circ y))\circ (R_1\circ z)\\
& =  R_1\circ ((x\circ y)\circ z)\\
& =  R_1\circ ( x \theta ^S f^S(y,z)\circ (y\circ z))\\
& =  R_1\circ (x \theta ^S f^S(y,z))\circ R_1\circ (y\circ z)\\
& = R_1\circ (x \theta ^S f^S(y,z))\circ ((R_1\circ y)\circ(R_1\circ z))\\
& =  (R_1\circ x)\circ ((R_1\circ y) \circ (R_1\circ z))\\ 
& ~~~~ (\text{since} ~(x,x \theta ^S f^S(y,z))\in  R).
 \end{align*}
Hence $S/R$ is a group.
Let $\phi:S \rightarrow S/{R}$ be the quotient homomorphism ($\phi(x)=R_1\circ x,~x\in S$). Let $H$ be a group with a homomorphism (of right loops) $\phi'~: ~ S\rightarrow H$. Since $\phi'$ is a homomorphism of right loops, $\phi'(S)$ is a sub right loop of $H$. Further, since the binary operation on $\phi'(S)$ is associative, it is a subgroup of $H$.
It is easy to verify that Ker$\phi'$ is an invariant sub right loop, that is there exists a congruence $K$ on $S$ such that $K_1=Ker\phi'$ (Lemma \ref{rps1}). By the Proposition \ref{prp1} there exists a unique one-one homomorphism $\bar{\phi}'~:S/K \rightarrow H$ such that $\bar{\phi}'\circ \nu={\phi}'$, where $\nu :S \rightarrow S/K$ is the quotient homomorphism ($\nu(x)=K_1\circ x,~x\in S$). Since $S/K$ is a group (being isomorphic to the subgroup ${\phi}'(S)$ of $H$), the associativity of its binary operation implies that $(x,x \theta ^S f^S(y,z))\in K$ for all $x,y,z \in S$. This implies $R \subseteq K$. This defines an onto homomorphism $\eta$ from $S/R$ to $S/K$ given by $\eta (R_1\circ x)=K_1\circ x$. Let ${\eta}'= \bar{\phi}'\circ \eta$. Then it follows easily that ${\eta}'$ is the unique homomorphism from $S/R$ to $H$ such that ${\eta}'\circ  \phi= {\phi}'$.
\end{proof}

\begin{rem}\label{gr3}
 Let $S$ and $R$ be as in the above Lemma. Let $U$ be a congruence
on $S$ containing $R$. Since $U_1/R_1$ is an invariant right subloop of $S/R_1$, it is a normal subgroup of $S/R_1$. Thus $S/U_1$ is a group for it is
isomorphic to $S/R_1/(U_1/R_1)$.
\end{rem}
It is easy to prove following lemma.
\begin{lem}\label{l3}
Let $(S, \circ)$ be a right loop. Let $L$ be the congruence on $S$ generated by $\{(x,x\theta^Sf^S(y,z))~|x,y\in S\} \cup \{(x \circ y, y\circ x) \mid x,y \in S\}$. Then $L$ is the smallest 
congruence on $S$ such that the quotient right loop $S/L$ is an abelian group.  
\end{lem}
 
\begin{lem}\label{a}
Let $G$ be a group, $H$ a subgroup of $G$ and $S$ a transversal of $H$ in $G$. Suppose that $N\trianglelefteq G $ containing $H$. Then 
$$ G/N=HS/N \cong S/{N\cap S}.$$
\end{lem}
\begin{proof} Suppose that $\circ$ denotes the induced right loop operation on $S$. Consider the map $\psi :S \rightarrow HS/N$ defined as $\psi(x)=xN$. This is a homomorphism, for
\begin{align*}
\psi(x\circ y) & = (x\circ y)N\\
~~~~~~~~~~& =  hxyN ~\text{for some $h \in H$}\\ 
~~~~~~~~~~& =  \psi(x) \psi(y)~~(H \subseteq N).
\end{align*}     
Also, $Ker \psi =\{x \in S |~xN=N \}
=S \cap N.$
Since for $h\in H~$ and $~x\in S,$~we have $hxN=xN$ and $\psi (x)=xN$, $\psi$ is onto and so by the Fundamental Theorem of homomorphisms for right loops,
$S/{N\cap S} \cong HS/N $.
\end{proof}
  Let $G$ be a group, $H$ a subgroup and $S$ a transversal of $H$ in $G$. Suppose that $\circ$ is the induced right loop structure on $S$. We define a map $f: S \times S \rightarrow H$ as: for $x,y \in S$, $f(x,y):=xy(x \circ y)^{-1}$. We further define the action $\theta$ of $H$ on $S$ as $\{x \theta h\}:= S \cap Hxh$ where $h \in H$ and $x \in S$. With these notations it is easy to prove following lemma. 
 
 \begin{lem}\label{l4}
 For $x,y,z \in S$, we have $x \theta^Sf^S(y,z)=x \theta f(y,z)$.
  \end{lem} 
 
\begin{lem}\label{b}
Let $H$ be a subgroup of a group $G$ and $S$ a transversal of $H$ in $G$. 
Let $U$ be the congruence on $S$ such that~$\{(x,x\theta h)~|~h \in H,x\in S \} \subseteq U $. Let $T$ be the equivalence class of $1$ under $U$. 
Then $S/U$ is a group. Moreover, $N=HT \trianglelefteq HS=G$ (and so $H \leq N$ and~ $N\cap S=T)$ and $G/N \cong S/U $.
\end{lem}

\begin{proof}~By Lemma \ref{l4}, $x\theta f(y,z)=x\theta^S f^S(y,z)$ for all $ x,y,z \in S$. Let $R$ be a congruence on $S$ generated by $\{(x,x \theta^S f^S(y,z))|x,y,z \in S\}$. Then, clearly $R \subseteq U$ and by Remark \ref{gr3}, $S/U$ is a group.
Let $\phi:G \rightarrow S/U$ be the map defined by $\phi(hx)=T\circ x,~h\in H,~x\in S$. This is a homomorphism, because for all $h_1,~h_2 \in H$ and $x_1,~x_2 \in S,$
\begin{align*} 
\phi(h_1x_1h_2x_2) & =  \phi(h_1h(x_1{\theta} h_2\circ x_2))~\text{for some $h \in H$}\\
 & =  T\circ (x_1{\theta} h_2\circ x_2)\\
 & =  (T\circ x_1)~\circ~(T\circ x_2)~~~~~~~(\text{for}~(x_1, x_1{\theta} h_2) \in U)\\
 & =  \phi (h_1x_1)\phi (h_2x_2).
\end{align*}
Let $h\in H$ and $x\in S.$ Then $hx\in Ker \phi $ if and only if $x\in T.$ Hence $Ker\phi = HT=N(say)$. This proves the lemma.
\end{proof}

\section{Solvable Right Loop}
In this section, we will define solvable right loop and show that if a group has a solvable generating transversal with respect to a core-free subgroup, then the group is solvable.

\begin{defn}
We define $S^{(1)}$ to be the smallest invariant subright loop of $S$ such that $S/S^{(1)}$ is an abelian group that is, if there is another invariant subright loop $N$ such that $S/N$ is an abelian group, then $S^{(1)} \subseteq N$.

We define $S^{(n)}$ by induction. Suppose $S^{(n-1)}$ is defined. Then $S^{(n)}$ is an invariant subright loop of $S$ such that $S^{(n)}=(S^{(n-1)})^{(1)}$.
\end{defn}

Note that if $S$ is a group, then $S^{(1)}$ is the commutator subgroup and $S^{(n)}$ is the $n$th-commutator subgroup.

\begin{defn}\label{s}
We call $S$ solvable if  there exists an $n \in \mathbb{N}$ such that $ S^{(n)}=\{1\}$.
\end{defn}
\begin{defn}
Let $S$ be a transversal of a subgroup $H$ of $G$. We call $S$ a solvable transversal if it is solvable with respect to the induced right loop structure.
\end{defn}
\begin{thm}\label{1}
If a group has a solvable generating transversal with respect to a core-free subgroup, then the group is solvable.
\end{thm}
\begin{proof}
Let $G$ be a group and $H$ a core-free subgroup of it. Suppose that $S$ is a generating transversal of $H$ in $G$. Then the group $G$ can be written as $HS$. By Lemma \ref{a}, $G/HG^{(1)} \cong S/S\cap HG^{(1)}$.
So, \begin{equation}\label{(i)} S^{(1)} \subseteq S\cap HG^{(1)}.
\end{equation}

By Lemma \ref{b}, $HS^{(1)}$ is a normal subgroup of $G$.
Thus $G/HS^{(1)} =S/S^{(1)} $ (Lemma \ref{a}).
Since $S/S^{(1)} $ is abelian, $ G^{(1)} \subseteq HS^{(1)}.$ Thus  
\begin{equation} \label{(ii)} S\cap HG^{(1)}\subseteq S^{(1)}.
\end{equation} 

From \eqref{(ii)} and \eqref{(i)}, it is clear that \begin{equation} S\cap HG^{(1)} = S^{(1)} \end{equation}\label{(iii)(a)} and
\begin{equation}\label{(iii)(b)}HG^{(1)} =HS^{(1)}.\end{equation} 

We will use induction to prove that, $HS^{(n)}=  H(HS^{(n-1)})^{(1)}$ for $n \geq 1$ and by $S^{(0)}$ we mean $S$. For $n=1$, $HS^{(1)}= HG^{(1)}= H(HS^{(0)})^{(1)}$ (by \eqref{(iii)(b)}).
By induction, suppose that $HS^{(n-1)}= H(HS^{(n-2)})^{(1)}$.
Since  $S^{(n-1)}/S^{(n)}\cong HS^{(n-1)}/HS^{(n)}  $ is an abelian group, $(HS^{(n-1)})^{(1)} \subseteq HS^{(n)}$. 
Thus $ H(HS^{(n-1)})^{(1)} \subseteq HS^{(n)}$.

 Further, $HS^{(n-1)}/H(HS^{(n-1)})^{(1)} $ $\cong S^{(n-1)}/(S^{(n-1)} \cap H(HS^{(n-1)})^{(1)})$ (Lemma \ref{a}). 
So $S^{(n)} \subseteq S^{(n-1)} \cap H(HS^{(n-1)})^{(1)}=S \cap H(HS^{(n-1)})^{(1)}$. 
That is, \begin{equation}\label{(*)} HS^{(n)}= H(HS^{(n-1)})^{(1)}~\text{for all}~ n \geq 1.\end{equation}  
Now  \eqref{(*)} implies that \begin{equation}\label{(**)} HS^{(n)}=H(HS^{(n-1)})^{(1)}  \supseteq H(H(HS^{(n-2)})^{(1)})^{(1)} \supseteq H(HS^{(n-2)})^{(2)} . \end{equation}  
Proceeding inductively, we have $ HS^{(n)}\supseteq H(HS)^{(n)}=HG^{(n)}$.
Suppose that $S$ is solvable, that is there exists $n \in \mathbb{N}$ such that $S^{(n)}=\{1\}$. Then $G^{(n)} \subseteq H$. Since $G^{(n)}$ is normal subgroup of $G$ contained in $H$, so $G^{(n)}=\{1\}$. This proves the theorem.
\end{proof}
 
 Converse of the above theorem is not true. For example take $G$ to be  the Symmetric group on three symbols and $H$ to be any two order subgroup of it. Then $H$ has no solvable generating transversal but we know that $G$ is solvable.  Following is the easy consequence of the above theorem.
 \begin{cor} \label{cor1}
 If $S$ is a solvable right loop, then $G_SS$ is solvable.
 \end{cor}

\section{Nilpotent Right Loop}

In this section, we define nilpotent right loop which is nothing but a particular case of nilpotent Mal'cev Varieties defined in \cite{smthmv}. We will show that if a group has a nilpotent generating transversal with respect to a core-free subgroup $H$, then $H$ is solvable group. This will generalize a result of \cite{rhb}.

\begin{defn}\label{6c2sd1}(\cite[Definition 211, p. 24]{smthmv}) Let $\beta$ and $\gamma$ be congruences on a right loop $S$. Let $(\gamma|\beta)$ be a congruence on $\beta$. Then $\gamma$ is said to \textit{centralize} $\beta$ by means of the \textit{centering congruence} $(\gamma|\beta)$ such that following conditions are satisfied:
 
\begin{enumerate}
	\item[(i)] $(x,y) ~(\gamma|\beta) ~(u,v) \Rightarrow x ~\gamma ~u$, for all $(x,y), (u,v) \in \beta$.
	\item[(ii)] For all $(x,y) \in \beta$, the map $\pi : (\gamma|\beta)_{(x,y)} \rightarrow \gamma_x$ defined by $(u, v)\mapsto u$ is a bijection, where for a set $X$ and an equivalence relation $\delta$ on $X$, $\delta_w$ denotes the equivalence class of $w \in X$ under $\delta$. 
	\item[(iii)] For all $(x,y) \in \gamma$, $(x,x) ~(\gamma|\beta) ~(y,y)$.
	\item[(iv)] $(x,y) ~(\gamma|\beta) ~(u,v)\Rightarrow (y,x) ~(\gamma|\beta) ~(v,u) $, for all $(x,y), (u,v) \in \beta$.
	\item[(v)] $(x,y) ~(\gamma|\beta) ~(u,v)$ and $(y,z) ~(\gamma|\beta) ~(v,w)$ $\Rightarrow (x,z) ~(\gamma|\beta) ~(u,w)$, for all $(x,y), (u,v), (y,z)$ and $(v,w)$ in $\beta$.
\end{enumerate}
  
 \end{defn}
 By $(i)$ and $(iv)$, we observe that $(x,y) ~(\gamma|\beta) ~(u,v) \Rightarrow y ~\gamma ~v$.
 
 Let $S$ be a right loop. If a congruence $\alpha$ on $S$ is centralized by $S \times S$, then it is called a \textit{central congruence} (see \cite[p. 42]{smthmv}). By \cite[Proposition 221, p. 34]{smthmv} and \cite[Proposition 226, p. 38]{smthmv}, there exists a unique maximal central congruence $\zeta(S)$ on $S$, called as the \textit{center congruence} of $S$. For a right loop, it is product of all centralizing congruences. The \textit{center} $\mathcal{Z}(S)$ of $S$ is defined as $\zeta_1$, the equivalence class of the identity $1$. In \cite[Proposition 3.3, p. 6]{vk}, it is observed that if $x \in \mathcal{Z}(S)$, then $x \circ (y \circ z)=(x \circ y) \circ z$ for all $y,z \in S$. In \cite[Proposition 3.4, p. 6]{vk}, it is observed that if $x \in \mathcal{Z}(S)$, then $x \circ y=y \circ x$ for all $y \in S$. This means that the center $\mathcal{Z}(S)$ is an abelian group.
 
\begin{defn}\label{nr}
A right loop $S$ is said to be \textit{nilpotent} if it has a central series 
\[\{1\}=\mathcal{Z}_0 \leq \mathcal{Z}_1 \leq \cdots \leq \mathcal{Z}_n=S\]
for some $n \in \mathbb{N}$, where \[\mathcal{Z}_{i+1}/\mathcal{Z}_i=\mathcal{Z}(S/\mathcal{Z}_i)~\text{and}~\mathcal{Z}_1=\mathcal{Z}(S).\]
\end{defn} 
On can observe that $\mathcal{Z}_i$ ($0 \leq i \leq n$) is an invariant right subloop of $S$. We call a transversal $S$ of a subgroup $H$ of a group $G$ to be nilpotent, if it is nilpotent with respect to the induced right loop structure.  

\begin{lem}
Every nilpotent right loop is solvable right loop. 
\end{lem}
\begin{proof}
Let $S$ be a nilpotent right loop with central series as follows:

\[\{1\}=\mathcal{Z}_0 \leq \mathcal{Z}_1 \leq \cdots \leq \mathcal{Z}_n=S\]
where \[\mathcal{Z}_{i+1}/\mathcal{Z}_i=\mathcal{Z}(S/\mathcal{Z}_i).\]
and $\mathcal{Z}_1$ is the center of $S$. Note that $S^{(1)}$ is contained in $\mathcal{Z}_{n-1}$ for $S/\mathcal{Z}_{n-1}$ is an abelian group. Thus $S^{(i)} \subseteq \mathcal{Z}_{n-i}$. This implies $S^{(n-1)} \subseteq \mathcal{Z}_{1}$. Since $\mathcal{Z}_{1}$ is an abelian group, so $S^{(n)}=\{ 1\}$. Thus $S$ is solvable.
\end{proof}

Since a nilpotent right loop $S$ is solvable, by Corollary \ref{cor1}, $G_SS$ is solvable. But in this proof we do not know much about the structure of $G_S$. We are going to prove again that if $S$ is nilpotent, then $G_S$ is solvable. As a consequence of it we obtain that if order of $S$ is prime power then the order of $G_SS$ will be prime power.  
\begin{prop}\label{4sp1}
Let $S$ be a right loop. Let $\theta : G_S \rightarrow G_{S/\mathcal{Z}(S)}$ be the onto homomorphism induced by natural projection $\nu:S \rightarrow S/\mathcal{Z}(S)$. Then $Ker \theta$ is isomorphic to a subgroup of abelian group $\prod_{\mathcal A}\mathcal{Z}(S) $ for some indexing set ${\mathcal A}$.
\end{prop}

\begin{proof}
Let ${\mathcal A}$ be a set obtained by choosing one element from each right coset of $\mathcal{Z}(S)$ in $S$ except $\mathcal{Z}(S)$. Then $S= \sqcup_{x_i \in \mathcal A}  (\mathcal{Z}(S) \circ x_{i})\sqcup \mathcal{Z}(S)$. Let $h \in Ker \theta$. Then $h(x_i)=z \circ x_i$ for some $z \in \mathcal{Z}(S)$. If $y=u_i \circ x_i$, where $u_i \in \mathcal{Z}(S)$, then \\
$h(y)=h(u_i \circ x_i)=u_i \circ h(x_i)$ (by condition $(C7)$ of \cite[Definition 2.1, p.71]{rltr} 

\qquad \qquad \qquad \qquad \qquad \qquad and \cite[Proposition 3.3, p.6]{vk})

\quad $=u_i \circ (z \circ x_i)$

\quad $=(u_i \circ z) \circ x_i$ (for $u_i \in \mathcal{Z}(S)$)

\quad $=(z \circ u_i) \circ x_i$ 

\quad $=z \circ (u_i \circ x_i)$ (for $z \in \mathcal{Z}(S) $)

\quad $=z \circ y$.

Thus $h \in Ker \theta$ is completely determined by $h(x_i)$ $(x_i \in {\mathcal A})$. Therefore, it defines a map $\eta: Ker \theta \rightarrow \prod_{\mathcal A}\mathcal{Z}(S) $ by $\eta(h)=(z_i)_{\mathcal A}$, where $h(x_i)=z_i \circ x_i$. One can check that $\eta$ is injective homomorphism.
\end{proof}
Following is the finite version of above proposition.

\begin{cor}
Let $S$ be a finite right loop with $|S/\mathcal{Z}(S)|=k$. Let $\theta : G_S \rightarrow G_{S/\mathcal{Z}(S)}$ be the onto homomorphism induced by natural projection $\nu:S \rightarrow S/\mathcal{Z}(S)$. Then $Ker \theta$ is isomorphic to a subgroup of abelian group $\mathcal{Z}(S) \times \cdots \times \mathcal{Z}(S)$ ($k-1$ times).
\end{cor}

\noindent Let $S$ be a nilpotent right loop with central series \begin{equation}\label{eq1} \{1\}=\mathcal{Z}_0 \leq \mathcal{Z}_1 \leq \cdots \leq \mathcal{Z}_n=S. \end{equation}
  
 Let $\theta_j : G_S \rightarrow G_{S/\mathcal{Z}_j}$ $(0 \leq j \leq n-1)$ be onto homomorphism induced by the natural projection $\nu_j : S \rightarrow S/\mathcal{Z}_j$. Then this will give a series \[\{1\}=Ker \theta_0 \leq \cdots \leq Ker \theta_{n-1}=G_S.\]
 
 Let $\theta : G_{S/\mathcal{Z}_j} \rightarrow G_{(S/\mathcal{Z}_j)/(\mathcal{Z}_{j+1}/\mathcal{Z}_j)}$ be onto homomorphism induced by the natural projection $\nu: S/\mathcal{Z}_j \rightarrow (S/\mathcal{Z}_j)/(\mathcal{Z}_{j+1}/\mathcal{Z}_j)$. By Proposition \ref{4sp1}, $Ker \theta$ is isomorphic to a subgroup of $\prod_{\mathcal B}\mathcal{Z}_{j+1}/\mathcal{Z}_j $ for some indexing set ${\mathcal B}$.
 
We now observe that each member of $Ker \theta_{j+1}/Ker \theta_j$ induces a member of $Ker \theta$. For this, we will see the action of an element of $Ker \theta_{j+1}/Ker \theta_j$ on the elements of $(S/\mathcal{Z}_j)/(\mathcal{Z}_{j+1}/\mathcal{Z}_j)$.  
 
 Let $h_{j+1}Ker \theta_j \in Ker \theta_{j+1}/Ker \theta_j$, where $h_{j+1} \in Ker \theta_{j+1}$ and $(\mathcal{Z}_{j+1}/\mathcal{Z}_j) \circ (\mathcal{Z}_j \circ x) \in (S/\mathcal{Z}_j)/(\mathcal{Z}_{j+1}/\mathcal{Z}_j)$. By the definition of $\theta_j$, each element of $Ker \theta_j$ acts trivially on the cosets of $\mathcal{Z}_j$. Since $G_{S/\mathcal{Z}_{j+1}} \cong G_{(S/\mathcal{Z}_j)/(\mathcal{Z}_{j+1}/\mathcal{Z}_j)}$, by definition of $\theta_{j+1}$, $h_{j+1}$ also acts trivially on $(\mathcal{Z}_{j+1}/\mathcal{Z}_j) \circ (\mathcal{Z}_j \circ x)$. Thus, we have proved following:
 
 \begin{prop}\label{4sp2}
Let $S$ be a nilpotent right loop with central series \ref{eq1}. Then there exists a series \[\{1\}=Ker \theta_0 \leq \cdots \leq Ker \theta_{n-1}=G_S\] such that $Ker \theta_{j+1}/Ker \theta_j$ is isomorphic to a subgroup of $\prod_{\mathcal B}\mathcal{Z}_{j+1}/\mathcal{Z}_j $ for some indexing set ${\mathcal B}$.
 \end{prop} 
 \begin{cor}
 Let $S$ be a nilpotent right loop. Then the group torsion $G_S$ is solvable group. 
 \end{cor}
\begin{proof}
By Proposition \ref{4sp2}, central series of $S$ gives a series  of $G_S$ with abelian quotient. 
\end{proof}
\begin{cor}
If a group $G$ has a nilpotent generating transversal with respect to a core-free subgroup $H$, then $H$ is solvable. 
\end{cor}
\begin{cor}
Let $S$ be a nilpotent generating transversal with respect to a core-free subgroup $H$ of a finite group $G$ such that $|S|=p^n$ for some prime $p$ and $n \in \mathbb{N}$. Then both $H$ and $G$ are $p$-groups.
\end{cor}

\section{Some Examples}
In this section, we will observe some examples and counterexamples. We have seen that concepts of solvability and nilpotencey of right loop can be transferred in term of a  generating transversal of a core-free subgroup of a group. There are examples of groups where no subgroup is core-free. Following is an example of such a group:

\begin{example}\label{ex2}
Consider the group $G= \langle x_1, x_2, x_3, x_4 | x_1^{p^n}= x_2^{p^3}= x_3^{p^2}= x_4^{p^2}= 1, [x_1, x_2] = x_2^{p^2},[x_1, x_3] = x_3^p, [x_1, x_4] = x_4^p, [x_2, x_3] = x_1^{p^{n-1}},[x_2, x_4] = x_2^{p^2}, [x_3, x_4] = x_4^p \rangle $ where $p$ is an odd prime and $n$ is the natural number greater than $2$. The above example has been taken from \cite{mv}. This is a nilpotent group of class 2 having no core-free subgroup.
\end{example}

 We now observe that a solvable right loop which is not a group need not be nilpotent right loop.
 
 \begin{example}
 Let $G=Alt(4)$, the alternating group of degree $4$ and $H=\{I,(1,2)(3,4)\}$, where $I$ denotes the identity permutation. Consider a right transversal $S=\{I,(1,3)(2,4),(1,2,3),(1,3,2),(2,3,4),(1,3,4)\}$ of $H$ in $G$. Note that $\langle S \rangle =G$ and $H$ is core-free. Then $G_SS \cong G$ and $G_S \cong H$. Also note that $S \cap N_G(H)=\{I,(1,3)(2,4)\}$, where $N_G(H)$ denotes the normalizer of $H$ in $G$. By \cite[Proposition 3.3, p. 6]{vk}, $\mathcal{Z}(S) \subseteq S \cap N_G(H)$. Let $\circ$ be the induced binary operation on $S$ as defined in the Section $1$. Observe that $(1,3)(2,4) \circ (1,3,4) \neq (1,3,4) \circ (1,3)(2,4)$. This implies that $\mathcal{Z}(S)=\{I\}$. Hence $S$ can not be nilpotent.
 
 Now by Lemma \ref{a}, $S/(S \cap N_G(H))$ is isomorphic to the cyclic group of order $3$. This implies that $S$ is solvable. 
 \end{example}
 Now, we observe that unlike the group the right loop of prime power order need not be nilpotent.
 
\begin{example}
Let $G =$ $$\langle (1,3)(2,4)(5,7,6,8),(1,4)(2,3)(5,8,6,7),(1,5)(2,6)(3,7)(4,8)\rangle \leq \text{Sym}(8),$$ where $\text{Sym}(n)$ denotes the symmetric group of degree $n$. Let $H$ be the stabilizer of $1$ in $G$. Consider $S=\{I,(1,2)(3,4),$ $(1,3)(2,4)(5,7,6,8),$ $(1,4)(2,3)$ $(5,8,6,7),$ $(1,5)(2,6)(3,7)(4,8),$ $(1,6)(2,5)(3,8)(4,7),$ $(1,7)(2,8)(3,6,4,5),$ $(1,8)(2,7)(3,5,4,6)\}$. Clearly $S$ is right transversal of $H$ in $G$. Note that the center $Z(G)=\{I,(1,2)(3,4)(5,6)(7,8)\}$ and $N_G(H)=HZ(G)$. Since $H$ is core-free and $\langle S \rangle=G$, $G \cong G_SS$ and $H \cong G_S$. Observe that $S \cap N_G(H)=\{I,(1,2)(3,4)\}$. By \cite[Proposition 3.3, p. 6]{vk}, $\mathcal{Z}(S) \subseteq S \cap N_G(H)$. Let $\circ$ be the induced binary operation on $S$ as defined in the section $1$. Observe that $(1,2)(3,4) \circ (1,5)(2,6)(3,7)(4,8) \neq (1,5)(2,6)(3,7)(4,8) \circ (1,2)(3,4)$. This implies that $\mathcal{Z}(S)=\{I\}$. Hence $S$ can not be nilpotent.
 
\end{example} 
Next, we will show that there are core-free subgroups of a nilpotent group which has none of its generating transversals nilpotent. But before proceeding to further examples, we need to prove following lemmas.
\begin{lem}\label{lem3} 
Suppose that $G$ is a nilpotent group of class $2$, $H$ is a core-free subgroup of $G$ and $S$ is generating transversal of $H$ in $G$. Then $\mathcal{Z}(G) \cap S=\mathcal{Z}(S)$.
\end{lem}
\begin{proof}
Take $x \in \mathcal{Z}(S)$. Then $x \circ y= y \circ x$ for all $y \in S$. This implies $xyx^{-1}y^{-1} \in H$ for all $y \in S$. Since group is nilpotent of class 2, so all commutators are central. For $H$ is core-free so $H$ will not contain any commutator element. This implies $xyx^{-1}y^{-1}=1$ or $xy=yx$ for all $ y \in S$. This proves that $ \mathcal{Z}(S) \subseteq \mathcal{Z}(G) \cap S$ (for $S$ generates $G$). Converse is obvious. This proves the lemma. 
\end{proof}

\begin{lem}\label{lem4}
For some prime $p$, suppose that $G$ is a $p$-group of nilpotent class $2$, $H$ is a core-free subgroup of $G$ and $S$ is generating transversal of $H$ in $G$. Then $\mathcal{Z}(G) \cap \Phi (G) \cap S=\{1\}$ where $\Phi (G)$ is the Frattini subgroup of $G$.
\end{lem}
\begin{proof}
Suppose that $1 \neq x \in \mathcal{Z}(G) \cap \Phi (G) \cap S$. Then by Lemma \ref{lem3}, $x \in \mathcal{Z}(S)$. Also $\mathcal{Z}(S)$ is an invariant right subloop, so $|\mathcal{Z}(S)|$ divides $|S|$. Consider $S^{\prime}=S\setminus \{x \} \cup \{hx\}$ for some $1\neq h \in H$. Note that $S^\prime$ also generates $G$. Then by Lemma \ref{lem3}, order of center of $S^\prime$ is one less than the order of center of $S$ and also   $|\mathcal{Z}(S^\prime)|$ divides $|S|$. This is not possible for order of $S$ is $p$ power. This proves the lemma.

\end{proof}

\begin{example}\label{ex1}
Consider the group $G= \langle x_1, x_2, x_3, x_4 | x_1^{p^n}= x_2^{p^2}= x_3^{p^2}= x_4^{p^4}= 1, [x_1, x_2] = x_2^{p},[x_1, x_3] = x_3^p, [x_1, x_4] = x_3^p, [x_2, x_3] = x_1^{p^{n-1}},[x_2, x_4] = x_2^p, [x_3, x_4] = 1 \rangle $ where $p$ is an odd prime. The above example has been taken from \cite{mv}. By the Lemma 2.1 of \cite{mv}, this group is nilpotent group of of class $2$ and its center and Frattini subgroup are equal. By Lemma \ref{lem3} and Lemma \ref{lem4}, it follows that center of each generating transversal is trivial. So none of the generating transversal is nilpotent.  

\end{example}

\vspace{0.2 cm}

\noindent \textbf{Acknowledgement:} Authors are thankful to Dr. R. P. Shukla, Department of Mathematics, University of Allahabad for the valuable discussion. First author is supported by UGC-BSR Research Start-up-Grant No. F.-2(20)/2012(BSR).

\end{document}